\documentclass[a4paper,10pt]{amsart}

\usepackage{amsmath}
\usepackage{amssymb}
\usepackage{dsfont}
\usepackage{enumerate}
\usepackage{eurosym}
\usepackage{picinpar}
\usepackage[all]{xy}

\renewcommand{\a}{\alpha}
\renewcommand{\b}{\beta}

\renewcommand{\l}{\lambda}

\newcommand{\z}{\mathds{Z}}
\newcommand{\n}{\mathds{N}}

\newcommand{\m}{{\mathfrak{M}}}
\newcommand{\B}{{\mathfrak{B}}}
                           
\newtheorem*{theorem-A}{Theorem A}
\newtheorem*{theorem-B}{Theorem B}
\newtheorem*{proposition*}{Proposition}

\newtheorem{theorem}{Theorem}[section]
\newtheorem{proposition}[theorem]{Proposition}
\newtheorem{lemma}[theorem]{Lemma}
\newtheorem{corollary}[theorem]{Corollary}

\theoremstyle{definition}
\newtheorem{defin}[theorem]{Definition}
\newtheorem{remark}[theorem]{Remark}
\newtheorem*{defin-bes}{Definition}

\newcommand{\opname}[1]{\operatorname{\rm{#1}}}

\newcommand{\Hom}{\opname{Hom}}
\newcommand{\Ext}{\opname{Ext}}

\newcommand{\End}{\opname{End}}
\newcommand{\Rep}{\opname{Rep}}
\newcommand{\rep}{\opname{rep}}
\newcommand{\im}{\opname{im}}
\newcommand{\coker}{\opname{coker}}
\renewcommand{\ker}{\opname{ker}}

\newcommand{\udim}{\underline{\dim}\,}

\begin{document}

\title[Quiver representations of maximal rank type and applications]{Quiver representations of maximal rank type and an application to representations of a quiver with three vertices}

\author{Marcel Wiedemann}

\address{Department of Pure Mathematics, University of Leeds, Leeds LS2 9JT, U.K.}

\email{marcel@maths.leeds.ac.uk}

\begin{abstract}
We introduce the notion of ``maximal rank type'' for representations of quivers, which requires certain collections of maps involved in the  representation to be of maximal rank. We show that real root representations of quivers are of maximal rank type. By using the maximal rank type property and universal extension functors we construct all real root representations of a particular wild quiver with three vertices. From this construction it follows that real root representations of this quiver are tree modules. Moreover, formulae given by Ringel can be applied to compute the dimension of the endomorphism ring of a given real root representation.
\end{abstract}

\subjclass[2000]{16G20}

\maketitle

\setcounter{section}{-1}
\section{Introduction}
\label{introduction}

\noindent Throughout this paper we fix an arbitrary field $k$. Let $Q$ be a (finite) {\textit{quiver}}, i.e. an oriented graph with finite vertex set $Q_0$ and finite arrow set $Q_1$ together with two functions $h,t:Q_1 \to Q_0$ assigning head and tail to each arrow $a\in Q_1$. For $i\in Q_0$ we define the following sets $H^Q(i):=\{a\in Q_1: h(a)=i\}$ and $T^Q(i):=\{a\in Q_1: t(a)=i\}$. 

A \textit{representation} $X$ of $Q$ is given by a vector space $X_i$ (over $k$) for each vertex $i\in Q_0$ together with a linear map $X_a: X_{t(a)} \to X_{h(a)}$ for each arrow $a\in Q_1$.

\begin{defin-bes}[Maximal Rank Type]
A representation $X$ of $Q$ is said to be of {\textit{maximal rank type}}, provided it satisfies the following conditions:
\begin{enumerate}[(i)]
\item For every vertex $i\in Q_0$ and for every subset $A\subseteq H^Q(i)$ the map
\begin{eqnarray*}
     \bigoplus_{a\in A} X_{t(a)} \stackrel{(X_a)_a }{ \longrightarrow} X_{i}
\end{eqnarray*}
is of maximal rank.
\item For every vertex $i\in Q_0$ and for every subset $B\subseteq T^Q(i)$ the map
\begin{eqnarray*}
      X_{i} \stackrel{(X_b)_b }{ \longrightarrow} \bigoplus_{b\in B} X_{h(b)}
\end{eqnarray*}
is of maximal rank.
\end{enumerate}
\end{defin-bes}

Clearly not every representation of $Q$ is of maximal rank type. The following example shows that even indecomposable representations of $Q$ might not be of maximal rank type.
\begin{center}
$\xymatrix{ k \ar@<1.0ex>[r]^-1 \ar@<-1.0ex>[r]_-0 & k }$
\end{center}

However, if $k$ is algebraically closed, a general representation for a given dimension vector $d$ is of maximal rank type. In particular, real Schur representations have this property; but clearly not all real root representations are real Schur representations. 

Let $\a$ be a positive real root for $Q$. Recall that there is a unique indecomposable representation of dimension vector $\a$ (see Section \ref{background} for details).

The main result of this paper is the following.
\begin{theorem-A}
Let $Q$ be a quiver and let $\a$ be a positive real root for $Q$. The unique indecomposable representation of dimension vector $\a$ is of maximal rank type.
\end{theorem-A}

In the second part of this paper we use Theorem A to construct all real root representations of the quiver
\begin{center}

$Q(f,g,h)$: 		$\xymatrix{ 1 \ar@<1.5ex>[r]^{\l_1} \ar@<-1ex>[r]^*-<0.4pc>{\vdots}_{\l_f} & 2 \ar@/^1.8pc/[r]^{\mu_1}
				         \ar@/^0.9pc/[r]^*-<0.4pc>{\vdots}_{\mu_g} & 3 \ar@/^1.8pc/[l]^{\nu_h} \ar@/^0.9pc/[l]^*-<1.3pc>{\vdots}_{\nu_1}}$,

\end{center}
with $f,g,h\ge 1$.

The quiver $Q(1,1,1)$ is considered by Jensen and Su in \cite{jensen}, where all real root representations are constructed explicitly. In \cite{ringel_2} Ringel extends their results to the quiver $Q(1,g,h)\; (g,h\ge 1)$ by using universal extension functors. In this paper we consider the general case and obtain the following result.
\begin{theorem-B}
Let $\a$ be a positive real root for the quiver $Q(f,g,h)$. The unique indecomposable representation of dimension vector $\a$ can be constructed by using universal extension functors starting from simple representations and real Schur representations of the quiver $Q'(f)\; (f\ge 1)$, where $Q'(f)$ denotes the following subquiver of $Q(f,g,h)$
\begin{center}

$Q'(f)$: $\xymatrix{ 1 \ar@<1.5ex>[r]^{\l_1} \ar@<-1ex>[r]^*-<0.4pc>{\vdots}_{\l_f} & 2}$.

\end{center}
\end{theorem-B}

The paper is organized as follows. In Section \ref{background} we discuss further notation and background results. In Section \ref{main-result} we prove Theorem A after discussing the constructions needed for the proof. To prove that real root representations are of maximal rank type, we have to show that certain collections of maps have maximal rank. The main idea of the proof is to insert an extra vertex and to attach to it the image of the considered map. Analyzing this modified representation yields the desired result.

In Section \ref{applications} we use the maximal rank type property of real root representations to prove Theorem B. It follows that real root representations of $Q(f,g,h)$ are tree modules. Moreover, using the formulae given in \cite{ringel} (Section 1) we can compute the dimension of the endomorphism ring for a given real root representation.

{\bf{Acknowledgements.}} The author would like to thank his supervisor, Prof. W. Crawley-Boevey, for his continuing support and guidance, especially for the help and advice he has given during the preparation of this paper. The author also wishes to thank the University of Leeds for financial support in form of a University Research Scholarship.

\section{Further notation and background results}
\label{background}
Let $Q$ be a quiver with vertex set $Q_0$ and arrow set $Q_1$. Let $X$ and $Y$ be two representations of $Q$. A \textit{homomorphism} $\phi: X \to Y$ is given by linear maps $\phi_i: X_i\to Y_i$ such that for each arrow $a\in Q_1$, $a:i\to j$ say, the square

\begin{center}
$\xymatrix{ X_i \ar@<0.0ex>[r]^-{X_a} \ar@<0.0ex>[d]^-{\phi_i} & X_j \ar@<0.0ex>[d]^-{\phi_j}	\\
					  Y_i \ar@<0.0ex>[r]^-{Y_a} & Y_j }$
\end{center}
commutes. The morphism $\phi$ is said to be an isomorphism if $\phi_i$ is an isomorphism for all $i\in Q_0$. The direct sum $X\oplus Y$ of two representations $X$ and $Y$ is defined by
\begin{eqnarray*}
	(X\oplus Y)_i &=& X_i\oplus Y_i,\quad \forall\; i\in Q_0,	\\
	(X\oplus Y)_a &=& \left( \begin{array}[2]{cc}
				   										X_a	& 0	\\
															0		& Y_a
						  						 \end{array} \right),\quad \forall\; a\in Q_1.
\end{eqnarray*}
A representation $Z$ is called \textit{decomposable} if $Z\cong X\oplus Y$ for non-zero representations $X$ and $Y$. In this way one obtains a category of representations, denoted by $\Rep_k Q$.

A \textit{dimension vector} for $Q$ is given by an element of $\n^{Q_0}$. We will write $e_i$ for the coordinate vector at vertex $i$ and by $d[i],\, i\in Q_0,$ we denote the $i$-th coordinate of $d\in \n^{Q_0}$. A dimension vector $d\in\n^{Q_0}$ is said to be \textit{sincere} provided $d[i]>0$ for all $i\in Q_0$. If $X$ is a finite dimensional representation, meaning that all vector spaces $X_i\; (i\in Q_0)$ are finite dimensional, then $\udim X= (\dim X_i)_{i\in Q_0}$ is the dimension vector of $X$. Throughout this paper we only consider finite dimensional representations. We denote by $\rep_k Q$ the full subcategory with objects the finite dimensional representations of $Q$.

The \textit{Ringel form} on $\z^{Q_0}$ is defined by
\begin{eqnarray*}
    \langle\a,\b\rangle = \sum_{i\in Q_0} \a[i]\b[i]   -   \sum_{a\in Q_1}\a[t(a)]\b[h(a)]
\end{eqnarray*}
Moreover, let $(\a,\b)=\langle\a,\b\rangle+\langle\b,\a\rangle$ be its symmetrization. 

We say that a vertex $i\in Q_0$ is \textit{loop-free} if there are no arrows $a:i\to i$. By a quiver without loops we mean a quiver with only loop-free vertices. In this paper we only consider quivers without loops. For a loop-free vertex $i\in Q_0$ the simple reflection $s_i:\z^{Q_0}\to \z^{Q_0}$ is defined by
\begin{eqnarray*}
	s_i(\a):=\a-(\a,e_i)e_i.
\end{eqnarray*}

A \textit{simple root} is a vector $e_i$ for $i\in Q_0$. The set of simple roots is denoted by $\Pi$. The \textit{Weyl group}, denoted by $W$, is the subgroup of $\textrm{GL}(\z^n)$, where $n=|Q_0|$, generated by the $s_i$. By $\Delta^+_{\textrm{re}}(Q) := \{\a \in W(\Pi) : \a> 0\} $ we denote the set of (positive) \textit{real roots} for $Q$. Let $M:=\{\b\in \n^{Q_0}: \,\textrm{$\b$ has connected support and}\, (\b,e_i)\le 0\; \textrm{for all $i\in Q_0$} \}$. By $\Delta^+_{\textrm{im}}(Q) := \bigcup_{w\in W}\, w(M)$ we denote the set of (positive) \textit{imaginary roots} for $Q$. Moreover, we define $\Delta^+ (Q):= \Delta^+_{\textrm{re}}(Q) \cup \Delta^+_{\textrm{im}}(Q)$. We have the following lemma.

\begin{lemma}[\cite{kac}, Lemma 2.1]
\label{root-lemma}
For $\a \in \Delta^+(Q)$ one has
\begin{enumerate}[(i)]
\item $\a \in \Delta^+_{\textrm{re}}(Q)$ if and only if $\langle \a,\a \rangle =1$,
\item $\a \in \Delta^+_{\textrm{im}}(Q)$ if and only if $\langle \a,\a \rangle \le 0$.
\end{enumerate}
\end{lemma}

As mentioned in the introduction we have the following remarkable theorem.

\begin{theorem}[Kac(\cite{kac}, Theorem 1 \& 2), Schofield(\cite{schofield}, Theorem 9)]
\label{kac-theorem}
Let $k$ be a field, $Q$ be a quiver, and let $\a \in \n^{Q_0}$.
\begin{enumerate}[(i)]
\item For $\a \notin  \Delta^+(Q)$ all representations of $Q$ of dimension vector $\a$ are decomposable.
\item For $\a \in \Delta^+_{\textrm{re}}(Q)$ there exists one and only one indecomposable representation of dimension vector $\a$.
\end{enumerate}
\end{theorem}

For finite fields and algebraically closed fields the theorem is due to Kac (\cite{kac}, Theorem 1 \& 2). As pointed out in the introduction of \cite{schofield}, Kac's method of proof showed that the above theorem holds for fields of characteristic $p$. The proof for fields of characteristic zero is due to Schofield (\cite{schofield}, Theorem 9).

For a given positive real root $\a$ for $Q$ the unique indecomposable representation (up to isomorphism) of dimension vector $\a$ is denoted by $X_\a$. By a \textit{real root representation} we mean an $X_\a$ for $\a$ a positive real root. The simple representation at vertex $i\in Q_0$ is denoted by $S(i)$. By a \textit{simple representation} we always mean an $S(i)$ for some vertex $i\in Q_0$. A \textit{Schur representation} is a representation with $\End_{kQ} (X)=k$. By a \textit{real Schur representation} we mean a real representation which is also a Schur representation. A positive real root is called a \textit{real Schur root} if $X_\a$ is a real Schur representation. An indecomposable representation $X$ is called \textit{exceptional} provided $\Ext^1_{kQ}(X,X)=0$.

We finish this section with the following useful formula: if $X,Y$ are representations of $Q$ then we have 
\begin{eqnarray*}
	\dim \Hom_{kQ}(X,Y) - \dim \Ext^1_{kQ} (X,Y) = \langle \udim X, \udim Y \rangle.
\end{eqnarray*}
It follows that $\Ext^1_{kQ}(X_\a,X_\a)=0$ for $\a$ a real Schur root.

\section{Proof of Theorem A}
\label{main-result}

Let $Q$ be a quiver with vertex set $Q_0$ and arrow set $Q_1$. Moreover, let $i\in Q_0$ be a vertex of $Q$ and let $X$ be a representation of $Q$. Note, that we only consider quivers without loops. For a given subset $A\subseteq H^Q(i)$ we define the quiver $Q^i_A$ and the representation $X^i_A$  (of the quiver $Q^i_A$) as follows
\begin{eqnarray*}
     (Q^i_A)_0 := Q_0\, \dot\cup\, \{z\},\quad	(Q^i_A)_1 := (Q_1-A)\, \dot\cup\, \{\gamma_a : a\in A\}\, \dot\cup\, \{\delta\}
\end{eqnarray*}
with
\begin{eqnarray*}
	  t(\gamma_a) := t(a), & & h(\gamma_a) := z\quad	\forall\, a\in A,	\\
		t(\delta) := z, & & h(\delta) := i,
\end{eqnarray*}
(heads and tails for all arrows in $Q_1-A$ remain unchanged) and
\begin{eqnarray*}
     (X^i_A)_j := X_j\quad \forall\, j\in Q_0,\quad (X^i_A)_z := \im 
     \left(\bigoplus_{a\in A} X_{t(a)} \stackrel{(X_a)_a }{ \longrightarrow} X_{i} \right) \subset X_i,
\end{eqnarray*}
with maps
\begin{eqnarray*}
		  (X^i_A)_q &:=& X_q											 \quad \forall\, q\in Q_1-A,	\\
		  (X^i_A)_{\delta} &:=& \textrm{inclusion},	\\
			(X^i_A)_{\gamma_a} &:=& \hat{X}_a							 \quad \forall\, a\in A,
\end{eqnarray*}
where $\hat{X}_a: X_{t(a)} \to (X^i_A)_z$ is the unique linear map with $(X^i_A)_{\delta}\circ \hat{X}_a = X_a $.

The construction above gives a functor $F^i_A:\rep_k Q \to \rep_k Q^i_A$, defined as follows
\begin{eqnarray*}
			F^i_A:\textrm{Ob}(\rep_k Q) &\to & \textrm{Ob}(\rep_k Q^i_A) \\
					X				&\mapsto &   X^i_A,
\end{eqnarray*}
with the obvious definition on morphisms. Moreover, there is a natural functor \linebreak ${}^i_A G:\rep_k Q^i_A \to \rep_k Q$, defined by
\begin{eqnarray*}
			{}^i_A G:\textrm{Ob}(\rep_k Q^i_A) &\to& \textrm{Ob}(\rep_k Q)	\\
						X				 &\mapsto & {}^i_A G(X),
\end{eqnarray*}
with
\begin{eqnarray*}
		({}^i_AG(X))_j	:= X_j\quad \forall\, j\in Q_0,	
\end{eqnarray*}
and maps
\begin{eqnarray*}
		({}^i_A G(X))_q	&:=& X_q\quad	\forall\, q\in Q_1-A,\\
		({}^i_A G(X))_a	&:=& X_{\delta} X_{\gamma_a}\quad \forall\, a\in A,
\end{eqnarray*}
together with the obvious definition on morphisms. The functor ${}^i_A G$ is left-adjoint to the functor $F^i_A$ and ${}^i_A G\circ F^i_A$ is naturally isomorphic to the identity functor on $\rep_k Q$. 

We get the following useful lemma.

\begin{lemma}
\label{lemma-indecomposable}
Let $Q$ be a quiver with vertex set $Q_0$ and arrow set $Q_1$. Moreover, let $i\in Q_0$ be a vertex and let $X$ be a representation of $Q$. If $X$ is indecomposable, then so is $F^i_A (X)= X^i_A$ for every subset of $A\subseteq H^Q(i)$.
\end{lemma}

\begin{proof}
Assume that $X^i_A = F^i_A (X) \cong U\oplus V$, then $X\cong {}^i_A G\circ F^i_A(X) \cong {}^i_A G(U) \oplus {}^i_A G(V)$. By assumption $X$ is indecomposable, so w.l.o.g. we can assume that ${}^i_A G(U)=0$. Hence,
\begin{eqnarray*}
	0 = \Hom_{kQ} ({}^i_A G(U) , X) = \Hom_{k Q^i_A} (U, F^i_A X) = \Hom_{k Q^i_A} (U,U\oplus V),
\end{eqnarray*}
which is only possible in case $U=0$. This proves the assertion.

\end{proof}

We are now able to proof the main theorem of this paper.

\begin{theorem-A}
\label{theorem-A}
Let $Q$ be a quiver and let $\a$ be a positive real root for $Q$. The unique indecomposable representation of dimension vector $\a$ is of maximal rank type.
\end{theorem-A}

\begin{proof}
Let $\a$ be a real root for $Q$ and let $X_\a$ be the unique indecomposable representation of $Q$ of dimension vector $\a$. Moreover, let $i\in Q_0$ and let $A\subset H^Q(i)$. We have to show that the map
\begin{equation*}
		\bigoplus_{a\in A} X_{t(a)} \stackrel{(X_a)_a }{ \longrightarrow} X_{i}
\end{equation*}
has maximal rank. This is equivalent to showing that 
\begin{equation*}
\dim (X^i_A)_z = \min\left\{\sum_{a\in A} \a[t(a)], \a[i]\right\}.
\end{equation*}
The representation $X^i_A$ of $Q^i_A$ is indecomposable by Lemma \ref{lemma-indecomposable}. It follows from Theorem \ref{kac-theorem} that $\udim X^i_A \in \Delta^+(Q^i_A)$. Hence, by Lemma \ref{root-lemma}, $\langle \hat{\a},\hat{\a} \rangle \le 1$, where $\hat{\a}:=\udim X^i_A$. We have 
\begin{eqnarray*}
	\langle \hat{\a},\hat{\a} \rangle &=& \underbrace{\langle \a,\a \rangle}_{=1} + \sum_{a\in A} \a[t(a)]\a[i] + \hat{\a}[z]^2-\hat{\a}[z]\hat{\a}[i]-\sum_{a\in A} \hat{\a}[t(\gamma_a)]\hat{\a}[z]	\\
	&=& 1 + \left(\hat{\a}[z] - \sum_{a\in A} \a[t(a)] \right)\cdot\left(\hat{\a}[z] -\a[i]  \right) \le 1
\end{eqnarray*}
and, hence,
\begin{eqnarray*}
	\left(\hat{\a}[z] - \sum_{a\in A} \a[t(a)] \right)\cdot\left(\hat{\a}[z] -\a[i]  \right) \le 0.
\end{eqnarray*}
However, we clearly have $\hat{\a}[z] \le \min \left\{ \sum_{a\in A} \a[t(a)], \a[i] \right\}$, by definition of $X^i_A$. This implies that
\begin{eqnarray*}
	\left(\hat{\a}[z] - \sum_{a\in A} \a[t(a)] \right)\cdot\left(\hat{\a}[z] -\a[i]  \right) = 0,
\end{eqnarray*}
i.e. $\hat{\a}[z] = \min \left\{ \sum_{a\in A} \a[t(a)], \a[i] \right\}$, and, hence, 
\begin{eqnarray*}
\dim (X^i_A)_z = \min\left\{\sum_{a\in A} \a[t(a)], \a[i]\right\}.
\end{eqnarray*}
This shows that the map $\bigoplus_{a\in A} X_{t(a)} \to X_{i}$ has maximal rank.

Dually, given subset $B\subset T^Q(i)$, to show that the map
\begin{eqnarray*}
      X_{i} \stackrel{(X_b)_b }{\longrightarrow} \bigoplus_{b\in B} X_{h(b)}
\end{eqnarray*}
has maximal rank it is equivalent to show that the map
\begin{eqnarray*}
      \bigoplus_{b\in B} X^*_{h(b)}  \stackrel{(X^*_b)_b }{\longrightarrow}			X^*_{i}
\end{eqnarray*}
has maximal rank, where ${}^*$ denotes the vector space dual. This follows from what we have proved above by considering the dual $X^*$ as a representation of the opposite quiver of $Q$.

\end{proof}

\section{Application(s): Representations of a quiver with three vertices}
\label{applications}

In this section we consider the quiver  

\begin{center}

$Q(f,g,h)$: 		$\xymatrix{ 1 \ar@<1.5ex>[r]^{\l_1} \ar@<-1ex>[r]^*-<0.4pc>{\vdots}_{\l_f} & 2 \ar@/^1.8pc/[r]^{\mu_1}
				         \ar@/^0.9pc/[r]^*-<0.4pc>{\vdots}_{\mu_g} & 3 \ar@/^1.8pc/[l]^{\nu_h} \ar@/^0.9pc/[l]^*-<1.3pc>{\vdots}_{\nu_1}}$,

\end{center}
with $f,g,h\ge 1$.

We define the following subquivers
\begin{center}

$Q'(f)$: $\xymatrix{ 1 \ar@<1.5ex>[r]^{\l_1} \ar@<-1ex>[r]^*-<0.4pc>{\vdots}_{\l_f} & 2 }$,

\end{center}
and
\begin{center}

$Q''(g,h)$: $\xymatrix{  2 \ar@/^1.8pc/[r]^{\mu_1}
				         \ar@/^0.9pc/[r]^*-<0.4pc>{\vdots}_{\mu_g} & 3 \ar@/^1.8pc/[l]^{\nu_h} \ar@/^0.9pc/[l]^*-<1.3pc>{\vdots}_{\nu_1}}$.

\end{center}

The quiver $Q(1,1,1)$ is considered by Jensen and Su in \cite{jensen}, where an explicit construction of all real root representations is given. Moreover, it is shown that all real root representations are tree modules and formulae to compute the dimensions of the endomorphism rings are given. In \cite{ringel_2} Ringel extends their results to the quiver $Q(1,g,h)\; (g,h\ge 1)$ by using the universal extension functors introduced in \cite{ringel}. 

In this section we consider the general case with $f,g,h\ge 1$. We use Ringel's universal extension functors to construct the real root representations of $Q=Q(f,g,h)$. 

We briefly discuss the situation for the subquivers $Q'(f)$ and $Q''(g,h)$: The real root representations of the subquiver $Q'(f)$ are preprojective or preinjective modules - for the path algebra $kQ'(f)$ - and can be constructed using BGP reflection functors (see \cite{bernstein}). It follows that the endomorphism ring of a real root representation of the subquiver $Q'(f)$ is isomorphic to the ground field $k$ and, hence, real root representations of $Q'(f)$ are real Schur representations. 

The subquiver $Q''(g,h)$ is considered by Ringel in \cite{ringel}. It is shown that all real root representations of $Q''(g,h)$ can be constructed using the universal extension functors defined in loc. cit., Section 1. Moreover, formulae to compute the dimensions of the endomorphism rings are given.

We see that the situation is very well understood for the subquivers $Q'(f)$ and $Q''(g,h)$. Therefore we will focus on real root representations with sincere dimension vectors.

\subsection{The Weyl group of $Q=Q(f,g,h)$}

Let $W$ be the Weyl group of $Q$. It is generated by the reflections $s_1,s_2,$ and $s_3$ subject to the following relations
\begin{eqnarray*}
	s_i^2 &=& 1, \quad i=1,2,3,	\\
	s_1s_3&=&s_3s_1,	\\
	s_1s_2s_1&=&s_2s_1s_2,\quad \textrm{if $f=1$}.
\end{eqnarray*}

We define the following elements of the Weyl group ($n\ge 0$)
\begin{eqnarray*}
		\zeta_1(n) &=& (s_1s_2)^n s_1,	\\
		\zeta_2(n) &=& (s_2s_1)^n s_2,  \\
		\rho_1(n)  &=& (s_1s_2)^n,	\\
	  \rho_2(n)  &=& (s_2s_1)^n,	
\end{eqnarray*}
and set $E:=\{\zeta_1(n), \zeta_2(n), \rho_1(n), \rho_2(n):n\ge 0\}$.

\begin{lemma}
\label{lemma-word-form}
Every element $w\in W-E$ can be written in the following form
\begin{equation}
	w = \chi_m s_3 \chi_{m-1} s_3 \chi_{m-2} s_3 \ldots s_3 \chi_2 s_3 \chi_1,	 \tag{$\ast$}\label{word-form}
\end{equation}
for some $m\ge 2$, where
\begin{eqnarray*}
	\chi_m &\in & \{\zeta_1(n) : n\ge 1\}\cup\{\zeta_2(n): n\ge 0\} \cup \{1\},	\\
	\chi_j &\in & \{\zeta_1(n) : n\ge 1\}\cup\{\zeta_2(n): n\ge 0\} ,\quad j=2,\ldots,m-1,	\\
	\chi_1 &\in & E.
\end{eqnarray*}
If $f=1$ then $w$ can be written in the form (\ref{word-form}) with only $\zeta_1(1)=\zeta_2(1), \rho_1(1),$ and $\rho_2(1)$ occurring.
\end{lemma}

\begin{proof}
Let $w\in W-E$. Clearly, we can write $w$ in the form
\begin{equation*}
	w = \chi'_m s_3 \chi'_{m-1} s_3 \chi'_{m-2} s_3 \ldots s_3 \chi'_2 s_3 \chi'_1,	
\end{equation*}
with $m\ge 2$, $\chi'_j\in E$ for $j=1,\ldots,m$, $\chi'_{m-1},\ldots,\chi'_2\notin \{1,s_1\}$, and $\chi'_m\ne s_1$. We modify the elements $\chi'_j$ to get a word of the form (\ref{word-form}). Let $2\le j\le m$; we consider five cases and modify $\chi'_j$ appropriately
\begin{enumerate}[(i)]
\item $\chi'_j=1$. We set $\chi_j:=\chi'_j$ and $\chi''_{j-1}=\chi'_{j-1}$. This case requires $j=m$.
\item $\chi'_j=\zeta_1(n)$ for $n\ge 1$. We set $\chi_j:=\chi'_j$ and $\chi''_{j-1}:=\chi'_{j-1}$.
\item $\chi'_j=\zeta_2(n)$ for $n\ge 0$. We set $\chi_j:=\chi'_j$ and $\chi''_{j-1}:=\chi'_{j-1}$.
\item $\chi'_j=\rho_1(n)$ for $n\ge 1$. We set $\chi_j:=\zeta_1(n)$ and $\chi''_{j-1}:=s_1\chi'_{j-1}$.
\item $\chi'_j=\rho_2(n)$ for $n\ge 1$. We set $\chi_j:=\zeta_2(n-1)$ and $\chi''_{j-1}:=s_1\chi'_{j-1}$.
\end{enumerate}
Now we have
\begin{eqnarray*}
	w &=&  \chi'_m s_3 \chi'_{m-1} s_3 \ldots s_3\chi'_j s_3 \chi'_{j-1}s_3\ldots s_3 \chi'_2 s_3 \chi'_1	\\
	  &=& \chi'_m s_3 \chi'_{m-1} s_3 \ldots s_3\chi_j s_3 \chi''_{j-1}s_3\ldots s_3 \chi'_2 s_3 \chi'_1,	
\end{eqnarray*}
with $\chi_j$ of the desired form and $\chi''_{j-1}\in E$. The result follows by descending induction on $j$.

\end{proof}

\begin{remark}
\begin{enumerate}[(i)]
\item For a given $w\in W$ the previous proof gives an algorithm to rewrite $w$ in the from (\ref{word-form}).
\item We make the following convention: in case $f=1$ we assume that $n\le 1$ in every occurrence of $\zeta_1(n), \zeta_2(n), \rho_1(n),$ and $\rho_2(n)$. Cases in which $n\ge 2$ is assumed do not apply to the case $f=1$.
\end{enumerate}
\end{remark}

\subsection{Universal Extension Functors}
\label{extension-functors}

In this section we recall some of the results from \cite{ringel} and prove key lemmas, which will be used in the next section to construct the real root representations of $Q$.

We fix a representation $S$ with $\End_{kQ} S=k$ and $\Ext^1_{kQ} (S,S)=0$. In analogy to loc. cit., Section 1, we define the following subcategories of $\rep_k Q$. Let $\m^S$ be the full subcategory of all modules $X$ with $\Ext^1_{kQ}(S,X)=0$ such that, in addition, $X$ has no direct summand which can be embedded into some direct sum of copies of $S$. Similarly, let $\m_S$ be the full subcategory of all modules $X$ with $\Ext^1_{kQ}(X,S)=0$ such that, in addition, no direct summand of $X$ is a quotient of a direct sum of copies of $S$. Finally, let $\m^{-S}$ be the full subcategory of all modules $X$ with $\Hom_{kQ} (X,S)=0$, and let $\m_{-S}$ be the full subcategory of all modules $X$ with $\Hom_{kQ} (S,X)=0$. Moreover, we consider
\begin{eqnarray*}
    \m_S^S = \m^S \cap \m_S,\quad   \m_{-S}^{-S}= \m^{-S} \cap \m_{-S}.
\end{eqnarray*}
According to loc. cit., Proposition 1 \& $1^*$ and Proposition 2, we have the following equivalences of categories
\begin{eqnarray*}
\overline{\sigma}_S&:& \m^{-S} \to \m^S/S,		\\
\underline{\sigma}_S&:& \m_{-S} \to \m_S/S,		\\
\sigma_S&:& \m_{-S}^{-S} \to \m^S_S/S,
\end{eqnarray*}
where $\m^{S}/S$ denotes the quotient category of $\m^{S}$ modulo the maps which factor through direct sums of copies of $S$, similarly for $\m_S/S$ and $\m^{S}_S/S$.

In the following, we briefly discuss how these functors and their inverses operate on objects. The functor $\overline{\sigma}_S$ is given by the following construction: Let $X\in\m^{-S}$ and let $E_1,\ldots,E_r$ be a basis of the $k$-vector space $\Ext^1_{kQ} (S,X)$. Consider the exact sequence $E$ given by the elements $E_1,\ldots,E_r$
\begin{eqnarray*}
		E: 0\to X\to Z\to \bigoplus_r S \to 0.
\end{eqnarray*}
According to loc. cit., Lemma 3, we have $Z\in \m^S$ and we define $\overline{\sigma}_S(X):=Z$. Now, let $Y\in\m_{-S}$ and let $E'_1,\ldots,E'_s$ be a basis of the $k$-vector space $\Ext^1_{kQ}(Y,S)$. Consider the exact sequence $E'$ given by $E'_1,\ldots,E'_s$
\begin{eqnarray*}
		E': 0\to \bigoplus_s S\to U\to Y\to 0.
\end{eqnarray*}
Then we have $U\in \m_S$ and we set $\underline{\sigma}_S(Y):=U$. The functor $\sigma_S$ is given by applying both constructions successively.

The inverse $\overline{\sigma}_S^{-1}$ is constructed as follows: Let $X\in \m^{S}$ and let $\phi_1,\ldots,\phi_r$ be a basis of the $k$-vector space $\Hom_{kQ}(X,S)$. Then by loc. cit., Lemma 2, the sequence
\begin{eqnarray*}
		0\to X^{-S}\to X \stackrel{(\phi_i)_i}{\longrightarrow} \bigoplus_r S\to 0
\end{eqnarray*}
is exact, where $X^{-S}$ denotes the intersection of the kernels of all maps $X\to S$. We set $\overline{\sigma}^{-1}_S (X):=X^{-S}$. Now, let $Y\in\m_{S}$. The inverse $\underline{\sigma}_S^{-1}$ is given by $\underline{\sigma}_S^{-1}(Y):=Y/Y'$, where $Y'$ is the sum of the images of all maps $S\to Y$. The inverse $\sigma^{-1}_S$ is given by applying both constructions successively.

Both construction show that
\begin{eqnarray*}
 		\udim \sigma^{\pm 1}_S(X) = \udim X - (\udim X, \udim S)\,\udim S.
\end{eqnarray*}
We have the following proposition.
\begin{proposition}[\cite{ringel}, Proposition 3 \& $3^*$]
Let $X\in \m^S_S$. Then
\begin{eqnarray*}
	\dim \End_{kQ} \sigma^{-1}_S(X) = \dim \End_{kQ}(X) - \langle \udim X, \udim S \rangle \cdot \langle \udim S,\udim X \rangle.
\end{eqnarray*}
Let $Y\in \m^{-S}_{-S}$. Then
\begin{eqnarray}
\label{dim-end}
	\dim \End_{kQ} \sigma_S(Y) = \dim \End_{kQ}(Y) + \langle \udim Y, \udim S \rangle \cdot \langle \udim S,\udim Y \rangle.
\end{eqnarray}
\end{proposition}

\begin{defin}
Let $\a$ be a real Schur root for $Q$. We define
\begin{eqnarray*}
	\m^{-\a}_{-\a}:=\m^{-X_\a}_{-X_\a}, \quad \m^{\a}_{\a}&:=&\m^{X_\a}_{X_\a},\quad	\textrm{and} \quad \sigma_\a:=\sigma_{X_\a}.
\end{eqnarray*}
\end{defin}

To construct real root representations of $Q$ we will reflect with respect to the following modules $S$: the simple representation $S(3)$ and the real root representations of $Q$ corresponding to certain positive real roots for the subquiver $Q'(f)$. Hence, we will use the following functors
\begin{eqnarray*}
\sigma_{e_3} &:&  \m_{-{e_3}}^{-{e_3}} \to \m_{e_3}^{e_3} / S(3)
\end{eqnarray*}
and
\begin{eqnarray*}
\sigma_{\chi} : \m_{-\chi}^{-\chi} \to \m_{\chi}^{\chi}/ X_{\chi}
\end{eqnarray*}
where $\chi$ denotes a positive real root for the subquiver $Q'(f)$. In order to use these functors, we have to make sure that $\sigma_{e_3}$ and $\sigma_{\chi}$ can be applied successively, i.e. we have to show that
\begin{eqnarray*}
    \m_{\chi}^{\chi} &\subset& \m_{-{e_3}}^{-{e_3}}, \\
    \m_{e_3}^{e_3} &\subset& \m_{-\chi}^{-\chi}.
\end{eqnarray*}
In general these inclusions do not hold. The following lemmas, however, show that under certain assumptions the functors can be applied successively.
We recall a key lemma from \cite{ringel}.
\begin{lemma}[\cite{ringel}, Lemma 4]
\label{lemma-ringel}
Let $S,T$ be modules, where $T$ is simple.
\begin{enumerate}[(i)]
\item If $\Ext^1_{kQ}(S,T)\ne 0$, then $\m^S \subset \m^{-T}$.
\item If $\Ext^1_{kQ}(T,S)\ne 0$, then $\m_S \subset \m_{-T}$.
\end{enumerate}
\end{lemma}

\begin{corollary}
\label{reflection-cor}
We have
\begin{eqnarray*}
    \m_{e_2}^{e_2} &\subset& \m_{-{e_3}}^{-{e_3}}, \\
    \m_{e_3}^{e_3} &\subset& \m_{-e_2}^{-e_2}.
\end{eqnarray*}
\end{corollary}

The previous corollary shows that $\sigma_{e_2}$ and $\sigma_{e_3}$ can be applied successively. In the following two lemmas we consider the situation when $\chi$ is a sincere real root for $Q'=Q'(f)$. The maximal rank type property of real root representations ensures that the situation is suitably well-behaved.
\begin{lemma}
\label{reflection-lemma-1}
Let $\chi$ be a sincere real root for $Q'$. Then we have $\m_{\chi}^{\chi} \subset \m_{-{e_3}}^{-{e_3}}$.
\end{lemma}
\begin{proof}
We have $\langle \chi,e_3 \rangle =-g\cdot \chi[2]<0$ and $\langle e_3,\chi \rangle = -h\cdot \chi[2]<0$. Thus, Lemma \ref{lemma-ringel} applies and we deduce $\m_{\chi}^{\chi} \subset \m_{-{e_3}}^{-{e_3}}$.

\end{proof}

\begin{lemma}
\label{reflection-lemma-2}
Let $\chi$ be a sincere real root for $Q'$ and let $Y\in \m_{e_3}^{e_3}-\{S(1)\}$ be a real root representation. Then  we have $Y\in \m^{-\chi}_{-\chi}$.
\end{lemma}

\begin{proof}
Let $Y\in \m_{e_3}^{e_3}-\{S(1)\}$ be a real root representation. Since \linebreak $\Ext^1_{kQ} (Y,S(3))=0=\Ext^1_{kQ} (S(3),Y)$ we get $\langle \udim Y, e_3\rangle \ge 0$ and $\langle e_3,\udim Y \rangle \ge 0$. This implies
\begin{eqnarray*}
\langle \udim Y,e_3 \rangle &=& -g\cdot \udim Y[2]+\udim Y[3]\ge 0,	\\
\langle e_3,\udim Y\rangle  &=& -h\cdot \udim Y[2] + \udim Y[3]\ge 0,
\end{eqnarray*}
and, thus, 
\begin{eqnarray*}
\udim Y[3]&\ge& g\cdot \udim Y[2],	\\
\udim Y[3]&\ge& h\cdot \udim Y[2],
\end{eqnarray*}
in particular, $\udim Y[3]\ge \udim Y[2]$. Since $\udim Y$ is a positive real root we can apply Theorem A, which implies that the maps $Y_{\mu_i}\; (i=1,\ldots,g)$ (of the representation $Y$) are injective and the maps $Y_{\nu_i}\; (i=1,\ldots,h)$ are surjective. 

Now, let $\phi: X_{\chi}\to Y$ be a morphism. Clearly, $\phi_3=0$. The injectivity of the maps $Y_{\mu_i}$ implies that $\phi_2=0$. This, however, implies that $\phi_1=0$ since otherwise the intersection of the kernels of the maps $Y_{\l_j}\; (j=1,\ldots,f)$ would be non-zero. This is nonsense since $Y$ is indecomposable and $Y\ne S(1)$. Hence, $\phi=0$.

Now, let $\psi:Y\to X_{\chi}$ be a morphism. Clearly, $\psi_3=0$. The surjectivity of the maps $Y_{\nu_i}$ implies that $\psi_2=0$. This, however, implies that $\psi_1=0$ since otherwise the intersection of the kernels of the maps $(X_{\chi})_{\l_j}\; (j=1,\ldots,f)$ would be non-zero. This is nonsense since $X_{\chi}$ is indecomposable and $\chi$ is sincere for $Q'$. Hence, $\psi=0$.

This completes the proof.

\end{proof}

The previous lemma shows the following: Let $X\in \m^{-e_3}_{-e_3}-\{S(1)\}$ be a real root representation, then we have $\sigma_{e_3} (X)\in \m^{-\chi}_{-\chi}$, where $\chi$ is a sincere real root for $Q'$.

\subsection{Construction of real root representations for $Q=Q(f,g,h)$}

In this section we construct the real root representations for $Q$ by using universal extension functors together with the results of the last section.

For $n\ge 1$ we define the functors
\begin{eqnarray*}
\sigma_{\zeta_1(n)} &:=& \begin{cases}
														\sigma_{\rho_1(\frac{n}{2})(e_1)},	&			\textrm{if $n$ is even,}	\\
														\sigma_{\zeta_1(\frac{n-1}{2})(e_2)},	&			\textrm{if $n$ is odd,}
												 \end{cases}
\end{eqnarray*}
and for $n\ge 0$ we define the functors
\begin{eqnarray*}
\sigma_{\zeta_2(n)} &:=& \begin{cases}
														\sigma_{\rho_2(\frac{n}{2})(e_2)},	&			\textrm{if $n$ is even,}	\\
														\sigma_{\zeta_2(\frac{n-1}{2})(e_1)},	&			\textrm{if $n$ is odd}.
												 \end{cases}
\end{eqnarray*}

\begin{remark}
For $n\ge 1$ we clearly have
\begin{enumerate}[(i)]
\item $\rho_1(n)(e_3) = \zeta_1(n) (e_3)$,
\item $\rho_2(n)(e_3) = \zeta_2(n-1) (e_3)$.
\end{enumerate}
\end{remark}

\begin{lemma}
\label{constr-first-part}
Let $\a$ be a positive non-simple real root of the following form:
\begin{enumerate}[(i)]
\item $\a=\chi(e_j)$ with $j\in\{1,2\}$ and $\chi\in E$.
\item $\a=\chi(e_3)$ with $\chi\in E$.
\end{enumerate}
Then the unique indecomposable representation of dimension vector $\a$ has the following properties:
\begin{enumerate}[(i)]
\item $X_\a$ is an indecomposable representation of the subquiver $Q'(f)$ and, hence, can be constructed using BGP reflection functors. Moreover, $\End_{kQ} X_\a =k$ and $X_\a\in \m^{-e_3}_{-e_3}$.
\item $X_\a$ can be constructed using the functors $\sigma_{\zeta_i(n)}\; (i=1,2)$ and $X_\a\in \m^{-e_3}_{-e_3}$.
\end{enumerate}
\end{lemma}

\begin{proof}
\begin{enumerate}[(i)]
\item Clear.
\item If $\a=\zeta_i(n)(e_3)\; (i=1,2)$ then $X_\a=\sigma_{\zeta_i(n)} S(3)$ and $X_\a\in\m^{-e_3}_{-e_3}$ by Lemma \ref{reflection-lemma-1} or Corollary \ref{reflection-cor} in case $\a=\zeta_2(0)$. If $\a=\rho_i(n)(e_3)\; (i=1,2)$ we use the previous remark to reduce to the case we have just considered.
\end{enumerate}

\end{proof}

We are now able to state and prove a more explicit version of Theorem B.

\begin{theorem}
Let $\a$ be a sincere real root for $Q$. Then $\a$ is of the form
\begin{enumerate}[(i)]
\item $\a=\zeta_i(n)(e_3)$ with $i\in\{1,2\}$ and $n\ge 1$, or
\item $\a=w(e_j)$ with $j\in\{1,2,3\}$ and $w=\chi_m s_3 \chi_{m-1} s_3 \chi_{m-2} s_3 \ldots s_3 \chi_2 s_3 \chi_1$ of the form (\ref{word-form}) with $\chi_1(e_j)\ne e_1$.
\end{enumerate}
The corresponding unique indecomposable representation of dimension vector $\a$ can be constructed as follows
\begin{enumerate}[(i)]
\item $X_{\zeta_i(n)(e_3)}=\sigma_{\zeta_i(n)} S(3)$,
\item $X_\a = \sigma_{\chi_m}\sigma_{e_3}\sigma_{\chi_{m-1}}\ldots \sigma_{\chi_2}\sigma_{e_3} X_{\chi_1(e_j)}$,
where $X_{\chi_1(e_j)}$ denotes the unique indecomposable of dimension vector $\chi_1(e_j)$: constructed in Lemma \ref{constr-first-part}.
\end{enumerate}
\end{theorem}

\begin{proof}
\begin{enumerate}[(i)]
\item Follows from Lemma \ref{constr-first-part}.
\item It follows from Lemma \ref{constr-first-part} that $X_{\chi_1(e_j)} \in \m_{-{e_3}}^{-{e_3}}$ and, hence, $\sigma_{e_3}$ can be applied. Moreover, by Corollary \ref{reflection-cor}, Lemma \ref{reflection-lemma-1}, and Lemma \ref{reflection-lemma-2} we have
\begin{eqnarray*}
	X_\b \in \m^{e_3}_{e_3}-\{S(1)\},\; \b\;\textrm{real root}	&\Longrightarrow& X_\b\in \m^{-\chi}_{-\chi},	\\
	\m_{\chi}^{\chi}	&\subset& \m_{-e_3}^{-e_3},
\end{eqnarray*}
where $\chi$ is a positive real root for the subquiver $Q'(f)$ not equal to $e_1$.
This completes the proof.
\end{enumerate}

\end{proof}

\begin{remark}
Using formula (\ref{dim-end}) together with Theorem B one can easily compute the dimension of the endomorphism ring of a sincere real root representation of $Q$.
\end{remark}

\subsection{Real root representations of $Q=Q(f,g,h)$ are tree modules}

In this section we show that real root representations of $Q=Q(f,g,h)$ are tree modules. We recall some definitions from \cite{ringel_3}. Let $Q$ be an arbitrary quiver with vertex set $Q_0$ and arrow set $Q_1$. Moreover, let $X\in \rep_k Q$ be a representation of $Q$ with $\udim X=d$. We denote by $\B_i$ a fixed basis of the vector space $X_i\, (i\in Q_0)$ and we set $\B=\cup_{i\in Q_0} \B_i$. The set $\B$ is called a basis of $X$. We fix a basis $\B$ of $X$. For a given arrow $a:i\to j$ we can write $X_a$ as a $d[j] \times d[i]$-matrix $X_{a,\B}$ with rows indexed by $\B_j$ and with columns indexed by $\B_i$. We denote by $X_{a,\B}(x,x')$ the corresponding matrix entry, where $x\in \B_i, x'\in \B_j$; the entries $X_{a,\B}(x,x')$ are defined by $X_a(x) = \sum_{x'\in \B_j} X_{a,\B}(x,x')\,x'$. The \textit{coefficient quiver} $\Gamma(X,\B)$ of $X$ with respect to $\B$ is defined as follows: the vertex set of $\Gamma(X,\B)$ is the set $\B$ of basis elements of $X$, there is an arrow $(a,x,x')$ between two basis elements $x\in \B_i$ and $x'\in \B_j$ provided $X_{a,\B}(x,x')\ne 0$ for $a:i\to j$.

\begin{defin}[Tree Module, see \cite{ringel_3}]
We call an indecomposable representation $X$ of $Q$ a \textit{tree module} provided there exists a basis $\B$ of $X$ such that the coefficient quiver $\Gamma(X,\B)$ is a tree.
\end{defin}

The following remarkable theorem is due to Ringel.

\begin{theorem}[\cite{ringel_3}]
\label{tree-module}
Let $k$ be a field and let $Q$ be a quiver. Any exceptional representation of $Q$ over $k$ is a tree module.
\end{theorem}

We briefly recall the construction of extensions of representations of quivers, as discussed in \cite{ringel_3} (Section 3) and \cite{ringel_4} (Section 2.1).

Let $Q$ be a quiver with vertex set $Q_0$ and arrow set $Q_1$. Moreover, let $X$ and $X'$ be representations of $Q$. The group $\Ext^1_{kQ}(X,X')$ can be constructed as follows: Let
\begin{eqnarray*}
	C^0(X,X')&:=& \bigoplus_{i\in Q_0} \Hom_k (X_i,X'_i),	\\
	C^1(X,X')&:=& \bigoplus_{a\in Q_1} \Hom_k (X_{t(a)},X'_{h(a)}).
\end{eqnarray*}
We define the map
\begin{eqnarray*}
	\delta_{XX'}:	C^0(X,X')	&\to & C^1(X,X')	\\
								(\phi_i)_i&\mapsto& (\phi_j X_a - X'_a \phi_i)_{a:i\to j}.
\end{eqnarray*}

The importance of $\delta_{XX'}$ is given by the following lemma.

\begin{lemma}[\cite{ringel_4}, Section 2.1, Lemma]
\label{lemma-ext-constr}
We have $\ker \delta_{XX'} = \Hom_{kQ}(X,X')$ and $\coker \delta_{XX'} = \Ext^1_{kQ} (X,X')$.
\end{lemma}

The following proof follows closely the arguments given in \cite{ringel_3} (Section 3 and Section 6).

\begin{lemma}
\label{lemma-univ-tree}
Let $Q$ be a quiver. Let $S$ be a representation with $\End_{kQ} S=k$ and $\Ext^1_{kQ} (S,S)=0$. Moreover, let $X\in \m^{-S}\; (\textrm{resp.},\; X\in \m_{-S})$ be a tree module. Then the representation $\overline{\sigma}_S(X)\; (\textrm{resp.},\; \underline{\sigma}_S(X))$ is a tree module. In particular, let $X\in\m^{-S}_{-S}$ be a tree module, then $\sigma_S(X)$ is a tree module.
\end{lemma}

\begin{proof}
We only consider the situation for the functor $\overline{\sigma}_S$. The situation for $\underline{\sigma}_S$ is analogous. Since $\sigma_S$ is given by applying $\overline{\sigma}_S$ and $\underline{\sigma}_S$ successively, the second assertion follows from the first.

We recall the construction of $\overline{\sigma}_S(X)$: Let $E_1,\ldots,E_r$ be a basis of the $k$-vector space $\Ext^1_{kQ} (S,X)$. Consider the exact sequence $E$ given by the elements $E_1,\ldots,E_r$
\begin{equation}
		E: 0\to X\to Z\to \bigoplus_r S \to 0,			 \tag{$+$}\label{ses-tree}
\end{equation}
then we have $\overline{\sigma}_S(X)=Z$. First of all, we note that $Z$ is indecomposable since \linebreak $\overline{\sigma}_S: \m^{-S}\to \m^S/S$ defines an equivalence of categories. Moreover, by Theorem \ref{tree-module} the representation $S$ is a tree module. Thus, we can chose a basis $\B_X$ of $X$ and a basis $\B_S$ of $S$ such that the corresponding coefficient quivers $\Gamma(X,\B_X)$ and $\Gamma(S,\B_S)$ are trees. We set $d_X:=\sum_{i\in Q_0} \dim X_i$ (dimension of $X$) and $d_S:=\sum_{i\in Q_0} \dim S_i$ (dimension of $S$). Since $X$ and $S$ are indecomposable representations the corresponding coefficient quivers are connected and, hence, $\Gamma(X,\B_X)$ has $d_X-1$ arrows and $\Gamma(S,\B_S)$ has $d_S-1$ arrows.

Let $a\in Q_1$. For given $1\le s\le t(a)$ and $1\le t\le h(a)$ we denote by \linebreak $M_{SX}(a,s,t) \in \Hom_k (S_{t(a)}, X_{h(a)})$ the matrix unit with entry one in the column with index $s$ and the row with index $t$ and zeros elsewhere. The set
\begin{eqnarray*}
			H_{SX}:=\{ M_{SX} (a,s,t): a\in Q_1, 1\le s\le t(a), 1\le t\le h(a)\}
\end{eqnarray*}
is clearly a basis of $C^1(S,X)$. Hence, we can choose a subset 
\begin{eqnarray*}
\Phi:=\{M_{SX} (a_i,s_i,t_i) : 1\le i\le r\} \subset H_{SX}
\end{eqnarray*}
such that $\Phi\oplus\im \delta_{SX} = C^1(S,X)$, which implies that the residue classes $\phi + \im \delta_{SX}\, (\phi\in\Phi)$ form a basis of $\Ext^1_{kQ}(S,X)$; these elements are responsible for obtaining the extension (\ref{ses-tree}).

We are now able describe the matrices of the representation $Z$ with respect to the basis $\B_X\cup \B_S$. Let $b\in Q_1$. The matrix $Z_b$ has the following form
\begin{eqnarray*}
	Z_b &=& \left[
							\begin{array}{cccc}
								X_b		&		N(b,1)		&	\hdots 	&	N(b,r)			\\
								      	& S_b				&					& 			\\
								      	&							&   \ddots&    	  \\
								      	&		      	  &         &	S_b				
							\end{array}
					 \right]
\end{eqnarray*}
with all other entries equal to zero and 
\begin{eqnarray*}
N(b,i) = \begin{cases}
						M(a_i,s_i,t_i),	&	\textrm{if $b=a_i$}	\\
						{\bf{0}},								& \textrm{otherwise},
				  \end{cases}
\end{eqnarray*}
where ${\bf{0}}$ denotes the zero matrix of the appropriate size. This explicit description allows us to count the overall number of non-zero entries in the matrices of the representation $Z$ with respect to the basis $\B_X\cup\B_S$: this number equals the number of arrows of the coefficient quiver $\Gamma(Z,\B_X\cup \B_S)$. We easily see that there are 
\begin{equation*}
(d_X-1)+r(d_S-1)+|\Phi| = d_X+rd_S-1 = \sum_{i\in Q_0} \dim Z_i-1
\end{equation*}
non-zero entries.

Now, since $Z$ is indecomposable, the coefficient quiver $\Gamma(Z,\B_X\cup\B_S)$ is connected and, hence, $\Gamma(Z,\B_X\cup\B_S)$ is a tree.

\end{proof}

The previous lemma and Theorem B give the following result.

\begin{proposition}
Let $\a$ be a positive real root for $Q=Q(f,g,h)\; (f,g,h\ge 1)$. Then the representation $X_\a$ is a tree module.
\end{proposition}

\begin{proof}
Representations of the subquiver $Q'=Q'(f)\; (f\ge 1)$ are exceptional representations, i.e. have no self-extensions, and, hence, are tree modules by Theorem \ref{tree-module}.

Now, let $X$ be a representation of $Q$ with $\udim X[3]\ne 0$. Then, by Theorem B (or the results in \cite{ringel} if $X$ is not sincere), $X$ can be constructed by using universal extension functors starting from a simple representation or a real root representation of the subquiver $Q'$: which is a tree module.

By Lemma \ref{lemma-univ-tree} the image of a tree module under the functor $\sigma_S$ is again a tree module. This proves the claim.

\end{proof}


\begin{thebibliography}{1}
		\bibitem{bernstein} 
		I.N. Bernstein, I.M. Gelfand and V.A. Ponomarev, `Coxeter functors and Gabriel's Theorem', {\em Russian Math. Surveys} 28 (1973) 17-32.
   	
   	\bibitem{jensen} 
   	B.T. Jensen and X. Su, `Indecomposable representations for real roots of a wild quiver', {\em Preprint} (2006).
   	
   	\bibitem{kac} 
   	V.G. Kac, `Infinite root systems, representations of graphs and invariant theory', {\em Inventiones mathematicae} 56 (1980) 57-92. 
    
    \bibitem{ringel} 
    C.M. Ringel, `Reflection functors for hereditary algebras', 
    {\em J. London Math. Soc.} 21 (1980) 465-479.
	  
	  \bibitem{ringel_2} 
	  C.M. Ringel, `The real root modules for some quivers', {\em Preprint} (2006).
	  
	  \bibitem{ringel_3} 
	  C.M. Ringel, `Exceptional modules are tree modules', {\em Linear Algebra Appl.} 275/276 (1998) 471-493.
	  
	  \bibitem{ringel_4} 
	  C.M. Ringel, `Representations of K-species and bimodules', {\em J. Algebra} 41 (1976) 269-302.
	  
    \bibitem{schofield} 
    A. Schofield, `The field of definition of a real representation of $Q$', {\em Proc. American Math. Soc.}
    116 (1992) 293-295.
\end{thebibliography}
\end{document}